\documentclass{amsart}

\usepackage{amsfonts}
\usepackage{amssymb}
\usepackage{amsthm}
\usepackage{eucal}
\usepackage{tikz}
\theoremstyle{plain}
\newtheorem{theorem}{Theorem}[section]
\newtheorem{lemma}[theorem]{Lemma}
\newtheorem{proposition}[theorem]{Proposition}
\newtheorem{corollary}[theorem]{Corollary}

\theoremstyle{definition}

\newtheorem*{acknowledgment}{Acknowledgment}

\newcommand{\F}{\mathfrak{F}}	
\newcommand{\Z}{\mathbb{Z}}
\newcommand{\gen}[1]{\langle #1\rangle}
\newcommand{\U}{\mathcal{U}}
\newcommand{\V}{\mathcal{V}}
\newcommand{\ol}[1]{\overline{#1}}
\newcommand{\normal}{\trianglelefteq}
\newcommand{\nnormal}{\ntrianglelefteq}

\newcommand{\ord}{\mathrm{ord}}
\newcommand{\Aut}{\mathrm{Aut}}

\newcommand{\lra}{\longrightarrow}
\newcommand{\G}{\Gamma}

\begin{document}
\title[Factorization graphs]{Factorization graphs of finite groups}

\author{M. Farrokhi D. G.} 
\address{Department of Mathematics, Institute for Advanced Studies in Basic Sciences (IASBS), and the Center for Research in Basic Sciences and Contemporary Technologies, IASBS, Zanjan 45137-66731, Iran}
\email{m.farrokhi.d.g@gmail.com\\farrokhi@iasbs.ac.ir}

\author{A. Azimi}
\address{Department of Mathematics, University of Neyshabur, Neyshabur, Iran}
\email{ali.azimi61@gmail.com}

\maketitle

\begin{abstract}
We introduce the factorization graph of a finite group and study its connectedness and forbidden structures. We characterize all finite groups with connected factorization graphs and classify those with connected bipartite factorization graphs. Also, we obtain a classification of all groups with claw-free, $K_{1,4}$-free, and square-free factorization graphs, excluding the non-solvable groups in the latter case.
\end{abstract}
\section{Introduction}
A group $G$ is \textit{factorized} if it is the product of two subgroups $A$ and $B$, namely $G=AB$, and this factorization is \textit{proper} if both $A$ and $B$ are proper subgroups of $G$. The group $G$ is called \textit{factorizable} if it has a proper factorization, and we say that it \textit{factorizes on a proper subgroup $A$} if $G=AB$ for some proper subgroup $B$ of $G$. Factorizations of groups appear naturally in many contexts, for instance products of groups, Frattini argument etc. The structure of subgroups in a factorization can influence the structure of the whole group. For example, a celebrated theorem of Ito \cite{ni} states that a group that is a product of two abelian groups is metabelian. Also, Kegel and Wielandt show that a finite group that is a product of two nilpotent subgroups is itself solvable, see \cite{ohk:1961,hw}.

The notion of factorization is generalized in various ways. One is the triple factorization, where a group $G$ is said to have a \textit{triple factorization} if $G=AB=BC=CA$ for some subgroups $A$, $B$, and $C$ of $G$. When conditions are applied to subgroups $A, B$, and $C$ in a triple factorization for $G$, results similar to those mentioned above are obtained; for example see \cite{ohk:1965}. 

The aim of this paper is to extend the above definitions and investigate factorizations from a combinatorial point of view by assigning to a group its graph of factorizations. The \textit{factorization graph} of a finite group $G$, denoted by $\F(G)$, is defined as a graph with vertex set 
\[V(\F(G))=\{H:H<G, H\neq G, H\nsubseteq\Phi(G)\}\]
with two vertices $U$ and $V$ adjacent if $G=UV$. Also, the \textit{proper factorization graph} $\F^*(G)$ of $G$ is the graph obtained from $\F(G)$ by removing all isolated vertices. Recall that  $\Phi(G)$ is the Frattini subgroup of $G$.

One can observe that proper factorizations of $G$ correspond to edges in the graph $\F(G)$. Also, every proper triple factorization of $G$ is equivalent to a triangle in its factorization graph. This suggests that we could consider every (induced) subgraph of $\F(G)$ as a (induced) factorization pattern in $G$. Recall that an \textit{induced subgraph} of a graph is a subgraph having all possible edges among its vertices. Also, by an \textit{induced factorization pattern} among a set of subgroups we mean a factorization pattern whose corresponding subgraph in $\F(G)$ is an induced subgraph of $\F(G)$.

Motivated by this, we are interested to see which finite groups cannot have particular (induced) factorization patterns as subgraphs of their factorization graphs. For instance, in Section 3, we shall describe all those factorizable groups having no factorization pattern that corresponds to an induced square, claw, or star graph with four pendants. 

Recall that the complete graph $K_n$ is a graph with $n$ pairwise adjacent vertices, and the complete bipartite graph $K_{m,n}$ is a graph with $m+n$ vertices partitioned into two sets $X$ and $Y$ with $m$ and $n$ elements, respectively, such that every edge has one end in $X$ and the other end in $Y$. The complete bipartite graphs $K_{1,n}$ are called \textit{star graphs}, and the star graph $K_{1,3}$ is known as the \textit{claw}. Also, a \textit{pendant vertex} is a vertex of degree one. The complement $\G^c$ of a graph $\G$ is a graph with the same vertices as $\G$ such that two vertices are adjacent in $\G^c$ if they are non-adjacent in $\G$. Let $\G_1,\ldots,\G_n$ be $n$ graphs, which we may assume to have disjoint vertex sets. The disjoint union $\G_1\cup\cdots\cup\G_n$ of graphs $\G_1,\ldots,\G_n$ is the graph whose vertex set and edge set is the union of vertex sets and edges sets of $\G_1,\ldots,\G_n$, respectively. If $\G_1=\cdots=\G_n$, then we write $n\G_1$ for their disjoint union (see \cite{dw}).

Let us call a group theoretical property is true on all non-superfluous subgroups of a finite group $G$ if it holds for all subgroups of $G$ not contained in the Frattini subgroup of $G$. Having this in mind, we begin our study in Section 2 by looking at a more restrictive class of groups, namely those finite groups $G$ factorizing on all non-superfluous proper subgroups. This enables us to give a classification of all finite groups with a connected factorization graph. Among these groups, we also classify those groups with no odd cycle as a factorization pattern, that is, groups whose factorization graphs are bipartite.

Throughout this paper, all groups are assumed to be finite. Given a group $G$, the minimum number of generators of $G$ is denoted by $d(G)$. Also, $\bar{\ }:G\lra G/\Phi(G)$ denotes the natural homomorphism. It is evident that $H\Phi(G)$ is a vertex of $\F(G)$ containing $\Phi(G)$ for every vertex $H$ of $\F(G)$, and that every neighbor of $H$ in $\F(G)$ is adjacent to $H\Phi(G)$ as well. In addition, $\F(G/\Phi(G))$ is isomorphic to an induced subgraph of $\F(G)$. We shall uses these facts without further references. All group-theoretical notation are standard and follow that of \cite{djsr}.
\section{Connectedness}
In this section, we shall consider those finite groups which factorize on all non-superfluous subgroups or equivalently those finite groups whose factorization graphs have no isolated vertices. Recall that a group $G$ is \textit{polycyclic} if it has a subnormal series all of whose factors are cyclic. Note that, in the case of finite groups, the class of polycyclic groups coincides with the class of solvable groups. We use the following two results in order to give a precise description of the groups under consideration.
\begin{proposition}\label{isolated-vertex-free-is-polycyclic}
Let $G$ be a finite group. If $\F(G)$ has no isolated vertices, then $G$ is solvable.
\end{proposition}
\begin{proof}
Clearly, we may assume that $\Phi(G)=1$. Let $|G|=p_1\cdots p_n$ with $1=p_0<p_1\leqslant\cdots\leqslant p_n$. Let $H_0=1$ and $K_0=G$. Suppose $H_j$ and $K_j$ are defined for all $0\leqslant j\leqslant i<n$ satisfying $|H_j|=p_j$, $|K_j|=p_{j+1}\cdots p_n$, $H_j\leqslant K_{j-1}$ and $[K_{j-1}:K_j]=p_j$ if $j\geq1$. Let $H_{i+1}$ be a subgroup of $K_i$ of order $p_{i+1}$. Since $H_{i+1}$ is not an isolated vertex, it follows that $G=H_{i+1}K^*_{i+1}$ for some proper subgroup $K^*_{i+1}$. Let $K_{i+1}=K^*_{i+1}\cap K_i$. Then $K_i=H_{i+1}K_{i+1}$ so that $K_{i+1}\normal K_i$ and $[K_i:K_{i+1}]=p_{i+1}$. Therefore, $G$ has a subnormal series
\[1=K_n\normal K_{n-1}\normal\cdots\normal K_1\normal K_0=G\]
with cyclic factors, as required.
\end{proof}

Notice that the above proposition extends a result of Hall \cite{ph} on complemented groups. Now, we can state and prove our classification of all finite groups with a connected factorization graph.
\begin{theorem}\label{connectedness}
Let $G$ be a finite group. Then the following statements are equivalent:
\begin{itemize}
\item[(1)]$\F(G)$ has no isolated vertices; 
\item[(2)]$\F(G)$ is connected;
\item[(3)]$\Phi(G)$ is the intersection of $m$ maximal subgroups of $G$ whose indexes in $G$ are the lowest $m$ prime divisors of $|G/\Phi(G)|$ for some $m$.
\end{itemize}
\end{theorem}
\begin{proof}
Let $G$ be a finite group and $\ol{G}:=G/\Phi(G)$. Also, let $|\ol{G}|=p_1\cdots p_n$, where $p_1\leqslant\cdots\leqslant p_n$ are primes.

It is evident that (2) implies (1). Also, if $\F(G)$ has no isolated vertices, then the subgroups $M_i:=K_i^*$ ($i=1,\ldots,n$) as in the proof of Proposition \ref{isolated-vertex-free-is-polycyclic} satisfy $[G:M_i]=p_i$ ($i=1,\ldots,n$) and $M_1\cap\cdots\cap M_n=\Phi(G)$, that is, (1) implies (3). Hence, we need only show that (3) implies (2). So, assume (3) holds. Then there exist maximal subgroups $M_1,\ldots,M_m$ of $G$ such that $M_1\cap\cdots\cap M_m=\Phi(G)$ and $[G:M_i]=p_i$, for $i=1,\ldots,m$. 

We show that 
\begin{itemize}
\item[(i)]$[M_{1^\pi}\cap\cdots\cap M_{i^\pi}:M_{1^\pi}\cap\cdots\cap M_{i^\pi}\cap M_{j^\pi}]=p_{j^\pi}$ for all $1\leq i<j\leqslant m$ and permutations $\pi\in S_m$,
\item[(ii)]$M_1\cap\cdots\cap M_i$ is a group of order $p_{i+1}\cdots p_m$, and
\item[(iii)]$M_1\cap\cdots\cap M_i\cap M_{i+1}$ is a maximal normal subgroup of $M_1\cap\cdots\cap M_i$,
\end{itemize}
for all $i=1,\ldots,m-1$. 

Let $\pi\in S_m$ be a permutation. Since $[H:H\cap K]\leqslant[G:K]$ for any two subgroups $H$ and $K$ of $G$, we get
\begin{align*}
|\ol{G}|&=[G:M_{1^\pi}][M_{1^\pi}:M_{1^\pi}\cap M_{2^\pi}]\cdots[M_{1^\pi}\cap\cdots\cap M_{(m-1)^\pi}:M_{1^\pi}\cap\cdots\cap M_{m^\pi}]\\
&\leqslant [G:M_{1^\pi}][G:M_{2^\pi}]\cdots[G:M_{m^\pi}]=p_1\cdots p_m,
\end{align*}
from which it follows that $m=n$ and 
\[[M_{1^\pi}\cap\cdots\cap M_{i^\pi}:M_{1^\pi}\cap\cdots\cap M_{i^\pi}\cap M_{(i+1)^\pi}]=p_{(i+1)^\pi}\]
for all $1\leqslant i<m$. This yields (i), and subsequently (ii). Also, since $M_1\cap\cdots\cap M_{i+1}$ is a subgroup of $M_1\cap\cdots\cap M_i$ of index $p_{i+1}$, the lowerst prime dividing $|M_1\cap\cdots\cap M_i|$, it follows that $M_1\cap\cdots\cap M_{i+1}$ is a normal subgroup of $M_1\cap\cdots\cap M_i$ proving (iii).

Now, let $H$ be a proper subgroup of $G$ not contained in $\Phi(G)$. We show that $H$ is adjacent to at least one of the maximal subgroups $M_1,\ldots,M_m$, from which the result will follow. 

If $H\nsubseteq M_1$, then obviously $H$ is adjacent to $M_1$. So, assume that $H\subseteq M_1$. Let $i$ be such that $H\subseteq M_1,\ldots,M_i$ and $H\nsubseteq M_{i+1}$. Then $H\nsubseteq M_1\cap\cdots\cap M_i\cap M_{i+1}$, which implies that 
\[M_1\cap\cdots\cap M_i=H(M_1\cap\cdots\cap M_i\cap M_{i+1})\]
for $M_1\cap\cdots\cap M_i\cap M_{i+1}$ is a maximal normal subgroup of $M_1\cap\cdots\cap M_i$ by (iii).

Suppose we have shown that $M_1\cap\cdots\cap M_j=H(M_1\cap\cdots\cap M_j\cap M_{i+1})$ for some $2\leqslant j\leqslant i$. Clearly, $M_1\cap\cdots\cap M_{j-1}\cap M_{i+1}\nsubseteq M_1\cap\cdots\cap M_j$ for otherwise $M_1\cap\cdots\cap M_j\cap M_{i+1}=M_1\cap\cdots\cap M_{j-1}\cap M_{i+1}$ contradicting (i). Now, since $M_1\cap\cdots\cap M_j$ is a maximal normal subgroup of $M_1\cap\cdots\cap M_{j-1}$ by (iii), it follows that 
\begin{align*}
M_1\cap\cdots\cap M_{j-1}&=(M_1\cap\cdots\cap M_j)(M_1\cap\cdots\cap M_{j-1}\cap M_{i+1})\\
&=H(M_1\cap\cdots\cap M_j\cap M_{i+1})(M_1\cap\cdots\cap M_{j-1}\cap M_{i+1})\\
&=H(M_1\cap\cdots\cap M_{j-1}\cap M_{i+1}).
\end{align*}
Hence, an inductive argument shows that $M_1=H(M_1\cap M_{i+1})$. But then 
\[G=M_1M_{i+1}=H(M_1\cap M_{i+1})M_{i+1}=HM_{i+1}\]
so that $H$ is adjacent to $M_{i+1}$, as required.
\end{proof}

We conclude this section by giving a complete classification of all finite groups with a bipartite factorization graph having no isolated vertex. Note that, by Theorem \ref{connectedness}, the corresponding factorization graphs are always connected. Indeed, as we shall see, the groups under investigation have complete bipartite factorization graphs. 
\begin{proposition}[\cite{mfdg}]\label{G/Phi(G)=metacyclic}
Let $G$ be a finite group such that $G/\Phi(G)$ is a metacyclic group of the form
$\gen{a,b:a^m=b^n=1,a^b=a^r}$ with $\gcd(m,n)=1$. Then $m$ is square-free and 
\[G=\gen{x,y:x^{mm'}=y^{nn'}=1,x^y=x^s},\]
where $m',n'$ divide some power of $m,n$, respectively, $s\equiv r\pmod m$ is an integer for which $n$ is the least multiple of $n^*$ satisfying $s^n\equiv1\pmod m$, and $n^*$ is the product of all prime divisors of $n$.
\end{proposition}
\begin{theorem}\label{bipartite}
Let $G$ be a finite group. Then $\F(G)$ is a bipartite graph with no isolated vertex if and only if $G$ is isomorphic to one of the following groups:
\begin{itemize}
\item[(1)]$C_{p^m}$, $C_{p^mq^n}$; or
\item[(2)]$\gen{x,y:x^{p^m}=y^{q^n}=1,x^y=x^\lambda}$ , 
\end{itemize}
where $p,q$ are distinct primes and $\ord_{p^m}\lambda=q$.
\end{theorem}
\begin{proof}
First suppose that $\F(G)$ is a bipartite graph with a bipartition $(\U,\V)$ having no isolated vertex. We have three cases:

Case 1. $d(G)\geq3$. If $x,y\in G\setminus\Phi(G)$ are such that $\gen{x}\in\U$ and $\gen{y}\in\V$, then $\gen{x,y}$ is adjacent to neighbors of $\gen{x}$ and $\gen{y}$, which implies that $\gen{x,y}\in \U\cap\V=\emptyset$, a contradiction. Hence, all cyclic subgroups belong to a single part, say $\U$. On the other hand, if $H$ is a vertex and $x\in H\setminus\Phi(G)$, then any neighbor of $\gen{x}$ is a neighbor of $H$ too. Thus $H\in\U$, which implies that $V(\F(G))=\U$, a contradiction.

Case 2. $d(G)=1$. Let $G=\gen{x_1}\times\cdots\times\gen{x_n}$, where $\gen{x_i}$ are the Sylow $p_i$-subgroups of $G$. Clearly, $\F(G)$ contains a triangle if $n\geq3$. Thus, we get $n\leqslant2$ and $\F(G)$ is either the empty graph with a single vertex if $n=1$, or it consists of two vertices and an edge between them if $n=2$.

Case 3. $d(G)=2$. Clearly, $\F(G)$ is not the null graph since every maximal subgroup $M$ of $G$ satisfies $\Phi(G)\subset M$ and hence is a vertex. First observe that if $H$ and $K$ are two subgroups not contained in $\Phi(G)$, with $H\in \U$ and $K\in\V$, then $H\cap K\subseteq\Phi(G)$; for otherwise $H\cap K$ is a vertex, which implies that every neighbor of $H\cap K$ is a neighbor of $H$ and $K$, a contradiction. Now, let $H,K,L\supset\Phi(G)$ be proper subgroups of $G$ satisfying $H\subseteq K\in\U$, and $L\in\V$ is adjacent to $H$. Clearly, $L$ is adjacent to $K$ so that $HL=KL=G$. Hence, $[H:H\cap L]=[K:K\cap L]$. Since $H\cap L=K\cap L=\Phi(G)$, we observe that $|H|=|K|$, that is, $H=K$. Thus, every subgroup of $\ol{G}:=G/\Phi(G)$ is maximal and a simple argument yields $\ol{G}\cong C_p\rtimes C_q$ for some primes $p$ and $q$. Clearly, $p\neq q$ for otherwise we have a triangle $\{H,K,L\}$ for any three distinct subgroups $\ol{H},\ol{K},\ol{L}$ of $\ol{G}$ of order $p$. Now, by Proposition \ref{G/Phi(G)=metacyclic},
\[G=\gen{x,y:x^{p^m}=y^{q^n}=1,x^y=x^\lambda},\]
where $m,n$ are positive integers and $\ord_{p^m}\lambda=q$.

Conversely, a simple verification shows that if $G$ is any of the groups listed in (1) or (2), then the set of vertices of $\F(G)$ partitions into two sets including subgroups of the forms $\gen{x^{pi}}\rtimes\gen{y}^{x^j}$ and $\gen{x}\rtimes\gen{y^{qk}}$ for some $i,j,k$, respectively, which implies that $\F(G)$ is a complete bipartite graph. The proof is complete.
\end{proof}
\begin{corollary}
Let $G$ be a finite group. Then $\F(G)$ is a bipartite graph with no isolated vertices if and only if it is a complete bipartite graph.
\end{corollary}
\section{Forbidden subgraphs}
This section is devoted to the study of the existence of various subgraphs of factorization graphs. Since there is nothing to mention for a non-factorizable group, all groups under consideration are assumed to have proper factorizations. We note that almost all finite groups are factorizable (see \cite{jk}), so our assumption is not restrictive. In what follows, a graph is said to be \textit{$\G$-free} if it has no induced subgraphs isomorphic to $\G$.
\begin{theorem}
Let $G$ be a finite factorizable group. Then $\F(G)$ is $K_{1,4}$-free if and only if $G$ is isomorphic to one of the following groups:
\begin{itemize}
\item[(1)]$C_{pqr}$, $C_p\times C_p$, $Q_8$, $C_4\times C_2$, $C_4\times C_4$;
\item[(2)]$C_{p^mq^n}$ for $1\leq m,n\leq 3$;
\item[(3)]$C_p\times Q_8$ or $C_p\times C_2\times C_2$ for $p>2$; or
\item[(4)]$\gen{x,y:x^{2^n}=y^3=1,y^x=y^{-1},(x^2)^y=x^2}$ for $n\leqslant3$,
\end{itemize}
where $p,q,r$ are distinct primes.
\end{theorem}

We prove the above theorem by breaking it up into two propositions separating the nilpotent and non-nilpotent cases.
\begin{proposition}\label{nilpotent k1,4-free}
Let $G$ be a finite nilpotent factorizable group. Then $\F(G)$ is $K_{1,4}$-free if and only if $G$ is isomorphic to one of the following groups:
\begin{itemize}
\item[(1)]$C_{pqr}$, $C_p\times C_p$, $Q_8$, $C_4\times C_2$, $C_4\times C_4$;
\item[(2)]$C_{p^mq^n}$ for $1\leq m,n\leq3$; or
\item[(3)]$C_p\times Q_8$ or $C_p\times C_2\times C_2$ for $p>2$,
\end{itemize}
where $p,q,r$ are distinct primes.
\end{proposition}
\begin{proof}
Assume $\F(G)$ is $K_{1,4}$-free. We have two cases to consider:

Case 1. $G$ is a $p$-group. If $d:=d(G)\geq3$, then $G/\Phi(G)=\gen{\ol{x}_1,\ldots,\ol{x}_d}$ for some $x_1,\ldots,x_d\in G$, and hence
\[\{\gen{\ol{x}_2,\ldots,\ol{x}_d},\gen{\ol{x}_1},\gen{\ol{x}_1\ol{x}_2},\gen{\ol{x}_1\ol{x}_3},\gen{\ol{x}_1\ol{x}_2\ol{x}_3}\}\]
induces a subgraph of $\F(G/\Phi(G))$ isomorphic to $K_{1,4}$, which is a contradiction. Thus, $d(G)\leqslant2$. If $d(G)=1$, then $\F(G)$ is a null graph, which is impossible by the hypothesis. Therefore, $d(G)=2$.

If $H$ is a vertex of $\F(G)$ with at least four conjugates in $G$ and $M$ is a maximal subgroup of $G$ not containing $H$, then $M$ along with four conjugates of $H$ induce a subgraph of $\F(G)$ isomorphic to $K_{1,4}$, which is a contradiction. Thus, $[G:N_G(H)]\leqslant3$ for all $H\in V(\F(G))$. We distinguish two cases:

Subcase 1.1. $G$ is abelian. Then $G=\gen{x}\times\gen{y}\cong C_{p^a}\times C_{p^b}$ for some $a\geq b\geq1$. If $a>1$, then 
\[\left\{\gen{x},\gen{y},\gen{x^py},\gen{x^{-p}y},\gen{x^p,y},\gen{x^{p^2},y}\right\}\]
induces a subgraph of $\F(G)$ isomorphic to $K_{1,4}$ or $K_{1,5}$ unless $p^a=4$. Therefore, either $G\cong C_p\times C_p$ whose graph is complete, or $G\cong C_4\times C_2$ or $C_4\times C_4$ whose graphs are complements of $2K_1\cup K_3$ and $3K_3$, respectively.

Subcase 1.2. $G$ is non-abelian. Clearly, $Z(G)\subseteq\Phi(G)$. Suppose first that $p\geq5$. Then every vertex of $\F(G)$ is a normal subgroup of $G$. Let $x,y\in G$ be such that $G=\gen{x,y}$ and $|x|+|y|$ is minimum. Since $\gen{x}$ and $\gen{y}$ are normal subgroups of $G$, it follows that $G$ is nilpotent of class $2$ for $[x,y]\in\gen{x}\cap\gen{y}\subseteq Z(G)$. Since $G$ is nonabelian we should have $\gen{x}\cap\gen{y}\neq1$. Then $x^{p^i}=y^{p^j\lambda}\neq1$ for some $\lambda$ coprime to $p$. Without loss of generality, assume that $i\leqslant j$. Then $(xy^{-p^{j-i}\lambda})^{p^i}=1$ so that $\{xy^{-p^{j-i}\lambda},y\}$ is a generating set of $G$ satisfying $|xy^{-p^{j-i}\lambda}|+|y|<|x|+|y|$ contradicting the assumption. Therefore, $p=2$ or $3$. 

We show that $G$ is nilpotent of class $2$. If not, $G$ has a vertex $H$, which is not a normal subgroup of $G$. Let $M$ be a maximal subgroup of $G$ not containing $H$. If $H$ has three distinct conjugates $H$, $H^a$, and $H^b$, then 
\[\{M,H,H^a,H^b,N_G(H)\}\]
induces a subgraph of $\F(G)$ isomorphic to $K_{1,4}$, which is a contradiction. Thus, $H$ has exactly two conjugates $H$ and $H^a$. If $K$ is a subgroup of $G$ such that $H\subset K\subset N_G(H)$ or $N_G(H)\subset K\subset G$, then again
\[\{M,H,H^a,K,N_G(H)\}\]
induces a subgraph of $\F(G)$ isomorphic to $K_{1,4}$, which is a contradiction. Also, if $H$ is non-cyclic and $x\in H\setminus M$, then
\[\{M,H,H^a,\gen{x},N_G(H)\}\]
induces a subgraph of $\F(G)$ isomorphic to $K_{1,4}$, which results in a contradiction. Thus, $H=\gen{x}$ is a cyclic maximal subgroup of $N_G(H)$, which itself is a maximal subgroup of $G$. By \cite[5.3.4]{djsr}, $N_G(H)=\gen{x}\rtimes\gen{y}$ is isomorphic to one of the groups $C_{2^{n-1}}\times C_2$, $D_{2^n}$, $Q_{2^n}$, $SD_{2^n}$, or $M_{2^n}$ for some $n$, where $SD_{2^n}$ and $M_{2^n}$ are defined as
\[SD_{2^n}=\gen{u,v:u^{2^{n-1}}=v^2=1,u^v=u^{2^{n-2}-1}}\]
and
\[M_{2^n}=\gen{u,v:u^{2^{n-1}}=v^2=1,u^v=u^{2^{n-2}+1}},\]
respectively.
Since $N_G(H)$ has two distinct cyclic maximal subgroups $H$ and $H^a$, it follows that $N_G(H)\cong C_{2^{n-1}}\times C_2$, $Q_8$, or $M_{2^n}$. The only non-abelian groups of order $16$ generated by two elements whose non-Frattini subgroups have at most two conjugates are 
\begin{align*}
G_1&=\gen{u,v:u^4=v^2=[v,u,u]=[v,u]^2=1},&G_2&=\gen{u,v:u^4=v^4=u^vu=1},\\
G_3&=\gen{u,v:u^8=v^2=u^vu^3=1},&G_4&=\gen{u,v:u^8=u^vu=1,u^4=v^2}.
\end{align*}
If 
\begin{align*}
V_1&=\{\gen{u,u^v},\gen{v},\gen{v^u},\gen{u^2,v},\gen{u^2,v^u}\},&V_2&=\{\gen{u,v^2},\gen{v},\gen{v^u},\gen{vu^{-1}},\gen{u^{-1}v}\},\\
V_3&=\{\gen{u},\gen{v},\gen{v^u},\gen{u^2v},\gen{u^2,v}\},&V_4&=\{\gen{u},\gen{v},\gen{v^u},\gen{u^{-1}v},\gen{vu^{-1}}\},
\end{align*}
then the subgraphs of $\F(G_i)$ induced by $V_i$ are isomorphic to $K_{1,4}$. Thus, $|G|>16$ and hence $N_G(H)\cong C_{2^{n-1}}\times C_2$ or $M_{2^n}$. It is easy to see that $\Phi(G)=\gen{x^2}\times\gen{y}$. Then $\gen{x^4}=\Phi(G)^2$ is a normal subgroup of $G$. Since $G'\nsubseteq\gen{x}$, we observe that $G/\gen{x^4}$ is a non-abelian group of order $16$ with all non-Frattini subgroups having at most two conjugates. Hence, $G/\gen{x^4}\cong G_i$ for some $1\leqslant i\leqslant 4$ so that $\F(G/\gen{x^4})$ has an induced subgraph isomorphic to $K_{1,4}$ contradicting the fact that $\F(G/\gen{x^4})$ is isomorphic to an induced subgraph of $\F(G)$ and $\F(G)$ is $K_{1,4}$-free. 

Therefore, $G$ is nilpotent of class $2$. Moreover, $G'$ is cyclic and $\exp(G')=\exp(G/Z(G))$. Clearly, $G/G'\cong C_p\times C_p$, $C_2\times C_4$, or $C_4\times C_4$ otherwise $\F(G/G')$, and hence $\F(G)$, would have an induced subgraph isomorphic to $K_{1,4}$ by Subcase 1.1. Then $G'\cong C_2$, $C_3$, or $C_4$ so that $|G|=8$, $16$, $27$, $32$, or $64$. A simple computation with GAP shows that $G\cong Q_8$ (see the codes after Corollary \ref{claw-free}).

Case 2. $G$ is not a $p$-group. If $G$ is non-cyclic, then we can write $G=H\times P$, where $P$ is a non-cyclic Sylow $p$-subgroup of $G$. Let $M_1,M_2,M_3$ be three distinct maximal subgroups of $P$. If $H$ is not a cyclic $q$-group of order $q$, then $H$ has a proper subgroup $K$. Then the subgraph induced by 
\[\{HM_1,M_2,M_3,KM_2,KM_3\}\]
is isomorphic to $K_{1,4}$, which is a contradiction. Thus, $H\cong C_q$ for some prime $q\neq p$. If $\F(P)$ has a claw as a subgraph, say $\{A,B,C,D\}$ with $A$ being adjacent to $B$, $C$, and $D$, then the subgraph induced by 
\[\{HA,P,B,C,D\}\]
is isomorphic to $K_{1,4}$, a contradiction. Therefore, $\F(P)$ is claw-free, which implies that $P\cong Q_8$ or $C_2\times C_2$ by Case 1. Therefore, $G\cong C_q\times Q_8$ or $C_q\times C_2\times C_2$. Finally, suppose that $G$ is cyclic and let $G=P_1\times\cdots\times P_n$ be the decomposition of $G$ into Sylow $p_i$-subgroups $P_i$ of $G$. If $n\geq4$, then the subgraph induced by 
\[\{P_1\ldots P_{n-1},P_n,P_1P_n,P_2P_n,P_3P_n\}\]
is isomorphic to $K_{1,4}$, which is a contradiction. Thus, $n\leqslant3$. Assume $n=3$. Then the subgraph induced by
\[\{P_iP_j,P_k,P_iP_k,P_jP_k,P_i^{p_i}P_k\}\]
is isomorphic to $K_{1,4}$ whenever $\{i,j,k\}=\{1,2,3\}$ and $|P_i|>p_i$, which is a contradiction. Thus, $\Phi(G)=1$ and so $G\cong C_{p_1p_2p_3}$.

Now, suppose $n=2$. If $|P_i|>p_i^3$ for some $i\in\{1,2\}$, then the subgraph induced by 
\[\{P_i,P_j,P_i^{p_i}P_j,P_i^{p_i^2}P_j,P_i^{p_i^3}P_j\}\]
is isomorphic to $K_{1,4}$ whenever $j\neq i$, which is a contradiction. Therefore, $|P_i|\leqslant p_i^3$ for $i=1,2$ and $G$ is isomorphic to a subgroup of $C_{p_1^3p_2^3}$. 

The converse is straightforward.
\end{proof}
\begin{proposition}\label{non-nilpotent k1,4-free}
Let $G$ be a finite non-nilpotent factorizable group. Then $\F(G)$ is $K_{1,4}$-free if and only if $G$ is isomorphic to
\[\gen{x,y:x^{2^n}=y^3=1,y^x=y^{-1},(x^2)^y=x^2}\]
for some $n\leqslant3$.
\end{proposition}
\begin{proof}
Assume $\F(G)$ is $K_{1,4}$-free. Since $\F(G)$ is non-empty, $G$ has a factorization into two maximal subgroups, say $M$ and $N$. If $[N:N_N(M)]\geq4$, then there exist $a,b,c\in N$ such that the subgraph induced by 
\[\{N,M,M^a,M^b,M^c\}\]
is isomorphic to $K_{1,4}$, which is a contradiction. So, $[G:N_G(M)]=[N:N_N(M)]\leqslant 3$. As a result, $G$ has normal maximal subgroups. Without loss of generality, we may assume that $M\normal G$.

Let $H$ be a non-normal maximal subgroup of $G$. Then $G=MH$ and like above, we conclude that $[G:N_G(H)]\leqslant3$, hence $[G:H]=3$. Let $H$, $H^a$, and $H^b$ be the distinct conjugates of $H$ in $G$. If $H$ has a subgroup $K$ such that $K\nsubseteq M$, then the subgraph induced by 
\[\{M,H,H^a,H^b,K\}\]
is isomorphic to $K_{1,4}$, which is a contradiction. Thus, every proper subgroup of $H$ is contained in $M$, from which it follows that $H$ is a cyclic $p$-group and $|G|=3p^n$ for some $n\geq1$. On the other hand, $|H/H_G|=2$ as $G/H_G$ embeds in $S_3$, hence $p=2$; there, $H_G$ is the core of $H$ in $G$ (the largest normal subgroup of $G$ contained in $H$). Now, since $G/C_G(H_G)$ is isomorphic to a subgroup of $\Aut(H_G)$ and $\Aut(H_G)$ is a $2$-group, it follows that $G=C_G(H_G)$ and consequently $H_G=Z(G)$.  Therefore, $G$ has the presentation
\[G=\gen{x,y:x^{2^n}=y^3=1,y^x=x^{2t}y^{-1},(x^2)^y=x^2}\]
for some integer $t$. From $y^x=x^{2t}y^{-1}$ in conjunction with the fact that $x^2$ and $y$ commute, $x$ is a $2$-element and $y$ is a $3$-element, it follows that $x^{2t}=1$ and $\gen{y}$ is a normal subgroup of $G$. Furthermore, $n\leqslant3$ otherwise the subgraph induced by 
\[\left\{\gen{x},\gen{y},\gen{y,x^{2^{n-1}}},\gen{y,x^{2^{n-2}}},\gen{y,x^{2^{n-3}}}\right\}\]
is isomorphic to $K_{1,4}$, which is a contradiction.

The converse is straightforward.
\end{proof}
\begin{corollary}\label{claw-free}
Let $G$ be a finite factorizable group. Then $\F(G)$ is claw-free if and only if $G$ is isomorphic to one of the following groups:
\begin{itemize}
\item[(1)]$C_{pqr}$, $C_p\times C_p$, $Q_8$;
\item[(2)]$C_{p^mq^n}$ for $m,n=1,2$; or
\item[(3)]$C_p\times C_2\times C_2$ for $p>2$,
\end{itemize}
where $p,q,r$ are distinct primes.
\end{corollary}

The following GAP function is used in Proposition \ref{nilpotent k1,4-free} to verify whether the factorization graph of a group $G$ is $K_{1,4}$-free.
\begin{verbatim}
IsK14Free:=function(G)
local V,HK;
 HK:=function(H,K)
  return Number(Union(List(K,k->RightCoset(H,k))))=Order(G);
 end;
 V:=Filtered(AllSubgroups(G),H->Order(H)<Order(G));
 return Number(Filtered(Combinations(V,5),S->\
    Set(List(S,H->Number(Filtered(S,K->HK(H,K)))))=[1,4]))=0;
end;
\end{verbatim}

To deal with the case of square-free graphs we need the following simple lemma. Recall that two subgroups $H$ and $K$ of a given group \textit{permute} if $HK=KH$. Also, a Frobenius group is said to be \textit{minimal} if it has no proper Frobenius subgroup.
\begin{lemma}\label{mutuallypermutingsubgroups}
If $G=H_1\cdots H_n$ ($n\geq4$) is a product of mutually permuting subgroups $H_i$ such that no proper subset of $\{H_1,\ldots,H_n\}$ generates $G$, then $\F(G)$ has an induced square.
\end{lemma}
\begin{proof}
Let $K=H_5\cdots H_n$. Then a simple verification shows that the subgraph induced by 
\[\{H_1H_2K,H_3H_4K,H_1H_2H_3K,H_1H_3H_4K\}\]
is an induced square, as required.
\end{proof}
\begin{theorem}\label{4-cycle}
Let $G$ be a finite factorizable group. If $\F(G)$ is square-free, then either $G$ contains a unique normal maximal subgroup $H$ such that
\begin{itemize}
\item[(i)]$H$ is perfect,
\item[(ii)]$\ol{H}$ is simple, 
\item[(iii)]$G=\gen{K,g}$ for all $g\in G\setminus H$ and every proper subgroup $K$ of $H$ on which $H$ factorizes, and
\item[(iv)]either $\ol{H}$ is the only nontrivial proper normal subgroup of $\ol{G}$, or $\ol{G}\cong\ol{H}\times C_p$ and $\ol{H}$ is non-factorizable;
\end{itemize}
or $G$ is isomorphic to one of the following groups:
\begin{itemize}
\item[(1)]$C_{pqr}$;
\item[(2)]$C_{p^kq}$ for $k\geq1$;
\item[(3)]$Q_8\times C_p$ for $p>2$;
\item[(4)]$C_p\times C_p\times C_q$ for $p\neq q$;
\item[(5)]a Frattini-by-(minimal Frobenius) group;
\item[(6)]a $p$-group $G$ generated by three elements such that $\gen{x,y}\leqslant G$ is maximal whenever $G=\gen{x,y,z}$ for some $z\in G$; or
\item[(7)]a non-cyclic $p$-group $G$ generated by two elements with cyclic Frattini subgroup.
\end{itemize}
Moreover, each of the groups in (1)--(7) has a square-free factorization graph.
\end{theorem}
\begin{proof}
Let $G=AB$ be a proper factorization of $G$. Clearly, $G=A^bB^a$ for all $a\in A$ and $b\in B$. If $A\neq A^b$ and $B\neq B^a$ for some $a\in A$ and $b\in B$, then $\{A,B,A^b,B^a\}$ induces a square, which is a contradiction. Hence, for every proper factorization $G=AB$ of $G$, wither $A\normal G$ or $B\normal G$.

On the other hand, if $M$ and $N$ are proper subgroups of $G$ which properly contain $A$ and $B$, respectively, then $\{A,B,M,N\}$ induces a square, which is a contradiction. Hence, for every proper factorization $G=AB$ of $G$, either $A$ is a maximal subgroup of $G$ or $B$ is a maximal subgroup of $G$.

From the above result one can deduce

($*$) for every proper factorization $G=AB$ of $G$, either $A$ is a normal maximal subgroup of $G$ or $B$ is a normal maximal subgroup of $G$.

Now, let $G=HK$ be a proper factorization of $G$. Then either $H$ or $K$, say $H$, is a normal maximal subgroup of $G$ by property ($*$). 

First suppose that $H=H'$ is perfect. Then $H$ is the unique normal maximal subgroup of $G$, for otherwise $G$ has a normal maximal subgroup $M$ different from $H$ so that $H'\subseteq M\cap H\subset H$, which is a contradiction. 

We show that $\ol{H}=H/\Phi(G)$ is simple. Suppose on the contrary that $\ol{H}$ is not simple. If $\ol{H}$ has a nontrivial proper characteristic subgroup $\ol{N}$, then $N\normal G$ and hence $G=NM$ for some maximal subgroup $M$ of $G$. Then $M\normal G$ by ($*$), which contradicts the uniqueness of the normal maximal subgroup $H$. Thus, $\ol{H}$ is characteristically simple, that is, $\ol{H}=\ol{S_1}\times\cdots\times\ol{S_k}$ for some isomorphic simple groups $\ol{S_1},\ldots,\ol{S_k}$. Let $p:=[G:H]$ and $x\in G\setminus H$. Since $\ol{x}$ does not stabilize $\ol{S_1},\ldots,\ol{S_k}$, we observe that $\{\ol{S_1},\ldots,\ol{S_k}\}=\{\ol{S_1},\ol{S_1}^{\ol{x}},\ldots,\ol{S_1}^{\ol{x}^{p-1}}\}$ by the uniqueness of the Remak decomposition (see \cite[3.3.12]{djsr}). Hence, $k=p$ and we may assume that $\ol{S_i}=\ol{S_1}^{\ol{x}^{i-1}}$, for $i=1,\ldots,p$. Let $\ol{x}^p=\ol{x_1}\cdots\ol{x_p}$ with $\ol{x_i}\in\ol{S_i}$, for $i=1,\ldots,p$. Then $(\ol{x}\ol{x_1}^{-1})^p=\ol{y_2}\cdots\ol{y_p}$ with $\ol{y_i}=\ol{x_i}\cdot\ol{x_1}^{-\ol{x}^{i-1}}\in\ol{S_i}$, for $i=2,\ldots,p$. Let 
\[\ol{K}:=\{\ol{g}\ol{g}^{\ol{x'}}\cdots\ol{g}^{\ol{x'}^{p-1}}:\ol{g}\in\ol{S}_1\}\]
be a subgroup of $\ol{H}$, where $x':=xx_1^{-1}$. Since $\ol{x'}^p$ centralizes $\ol{S}_1$, it follows that $\ol{K}^{\ol{x'}}=\ol{K}$ so that $\gen{\ol{K},\ol{x}}$ is a proper subgroup of $\ol{G}$. Thus, we have the proper factorization
\[\ol{G}=(\ol{S}_1\times\cdots\times\ol{S}_{p-1})\gen{\ol{K},\ol{x'}},\]
which implies that $\gen{K,x}$ is a normal maximal subgroup of $G$ different from $H$ by ($*$), a contradiction. Therefore, $\ol{H}$ is simple.

If $H=KL$ is a proper factorization of $H$ and $g\in G\setminus H$, then $G=\gen{K,g}$; otherwise $G=\gen{K,g}L$ is a proper factorization of $G$ contradicting ($*$) in conjunction with the fact that $H$ is the only normal maximal subgroup of $G$. Next assume that $N$ is a proper normal subgroup of $G$ such that $H\neq N\nsubseteq\Phi(G)$. Since $H\cap N\normal H$, it follows that $H\cap N\subseteq\Phi(G)$. If $H$ is factorizable, then $H=KL$ for some proper subgroups $K$ and $L$ of $H$ so that $G=(KN)(LN)$. But then either $KN$ or $LN$ is a normal maximal subgroup of $G$ different from $H$ by ($*$), which is a contradiction. Hence, either $\ol{H}$ is the unique nontrivial proper normal subgroup of $\ol{G}$, or $\ol{G}=\ol{H}\times\ol{N}$ with $\ol{N}$ a cyclic group of prime order and $\ol{H}$ is non-factorizable, as required.

Now, assume that $H\neq H'$. We show that $G$ is solvable. Let $M$ be a non-normal maximal subgroup of $G$ and $g\in G\setminus M$. If $H'\nsubseteq M$, then $\{H,M,H',M^g\}$ induces a square, which is a contradiction. Hence, $H'\subseteq M$ so that $H'$ is nilpotent by using \cite[Satz 16]{wg}. Therefore, $G$ is solvable as $G/H$ is cyclic of prime order, $H/H'$ is abelian, and $H'$ is nilpotent. Let $\pi(G)=\{p_1,\ldots,p_n\}$. Then, by \cite[9.2.1]{djsr}, $G$ has mutually permuting Sylow $p_i$-subgroups $P_i$ for $i=1,\ldots,n$. By Lemma \ref{mutuallypermutingsubgroups}, $n\leqslant3$. Now, We have three possibilities.

(1) $n=3$. Then $G=P_i(P_jP_k)$ for $\{i,j,k\}=\{1,2,3\}$ so that $P_jP_k$ is a maximal normal subgroup of $G$ by ($*$). Hence, $P_i\cong C_{p_i}$. Similarly, $P_iP_j,P_iP_k\normal G$, which implies that $P_i=P_iP_j\cap P_iP_k$ is a normal subgroup of $G$. Therefore, $G\cong C_{p_1p_2p_3}$.

(2) $n=2$. Since $G=P_1P_2$ we may assume without loss of generality that $P_1$ is a normal maximal subgroup of $G$. Then $P_2=\gen{y}\cong C_{p_2}$. First suppose that $P_2\nnormal G$. From \cite[Theorems 5.3.3 and 5.3.11]{hb}, we know that $\Phi(G)=\Phi(P_1)$ so that $\ol{P_1}$ is an elementary abelian $p_1$-group. 

We show that $\ol{P_1}$ is a minimal normal subgroup of $\ol{G}$. If not, $\ol{P_2}$ does not act irreducibly on $\ol{P_1}$ by conjugation. We claim that $\ol{P_1}$ has a nontrivial proper $\ol{P_2}$-invariant subgroup $\ol{A}$ such that $\ol{A}\ol{P_2}^{\ol{g}}\neq\ol{A}\ol{P_2}$ for some $\ol{g}\in\ol{G}$. If the claim is not true for $\ol{A}$, then for every $\ol{g}\in\ol{G}$ we must have $\ol{A}\ol{P_2}^{\ol{g}}=\ol{A}\ol{P_2}$ so that $\ol{P_2}^{\ol{g}}=\ol{P_2}^{\ol{a}}$ for some $\ol{a}\in\ol{A}$. Hence, $\ol{g}\ol{a}^{-1}\in N_{\ol{G}}(\ol{P_2})$, which implies that $\ol{G}=\ol{A}N_{\ol{G}}(\ol{P_2})$. As a result, $\ol{A^*}:=\ol{P_1}\cap N_{\ol{G}}(\ol{P_2})$ in a nontrivial proper subgroup of $\ol{P_1}$. Also, $[\ol{A^*},\ol{P_2}]\subseteq\ol{P_1}\cap\ol{P_2}=\ol{1}$, which implies that  $\ol{A^*}$ is a $\ol{P_2}$-invariant subgroup of $\ol{P_1}$. Since $\ol{A^*}N_{\ol{G}}(\ol{P_2})=N_{\ol{G}}(\ol{P_2})\neq\ol{G}$, from the above argument on $A$, we observe that there must exists an element $\ol{g}\in\ol{G}$ such that $\ol{A^*}\ol{P_2}^{\ol{g}}\neq\ol{A^*}\ol{P_2}$. Hence, we may replace $\ol{A}$ by $\ol{A^*}$ and assume that $\ol{A}\ol{P_2}^{\ol{g}}\neq\ol{A}\ol{P_2}$ for some $\ol{g}\in\ol{G}$. On the other hand, by utilizing the well-known theorem of Maschke (see \cite[8.1.2]{djsr}), it yields $\ol{P_1}=\ol{A}\times\ol{B}$, where $\ol{B}$ is a nontrivial $\ol{P_2}$-invariant subgroup of $\ol{P_1}$. But then 
\[\{P_1,AP_2,B,AP_2^g\}\]
induces a square, which is a contradiction. As a result, it yields $\ol{G}=\ol{P_1}\rtimes\ol{P_2}$ is a Frobenius group whose Frobenius kernel and complement are $\ol{P_1}$ and $\ol{P_2}$, respectively. Moreover, $\F^*(G)$ is a star graph with $P_1$ at the center and $P_2^gK$ are the pendants where $g\in G$ and $K\leqslant\Phi(G)$ permutes with $P_2^g$, that is, $P_2^gK=KP_2^g$.

Now, suppose that $P_2\normal G$. Then $G\cong P_1\times P_2$. Let $\{x_1,\ldots,x_m\}$ be a minimal generating set for $P_1$. Since $G=\gen{P'_1,x_1}\cdots\gen{P'_1,x_m}\gen{y}$, by Lemma \ref{mutuallypermutingsubgroups}, we should have $m=1$ or $2$. Assume $m=2$ and $P_1=\gen{x_1,x_2}$. If $P_1\neq\gen{x_1}\gen{x_2}$, then
\[\{\gen{P'_1,x_1x_2},\gen{x_1,y},P_1,\gen{x_2,y}\}\]
induces a square, which is a contradiction. Hence, $P_1=\gen{x_1}\gen{x_2}$ for each generating set $\{x_1,x_2\}$ of $P_1$. In particular, $\gen{x_1}$ or $\gen{x_2}$ is a maximal subgroup of $P_1$ otherwise there exist maximal subgroups $M_1$ and $M_2$ of $P_1$ containing properly $\gen{x_1}$ and $\gen{x_2}$, respectively, and hence 
\[\{\gen{x_1},\gen{x_2,y},M_1,M_2\gen{y}\}\]
induces a square, which is impossible. So, by \cite[5.3.4]{djsr}, either $P_1\cong C_{p_1}\times C_{p_1}$ or $Q_8$, or there exists $x_3\in P_1$ such that $P_1=\gen{x_1}\gen{x_3}\neq\gen{x_2}\gen{x_3}$. In the latter case, 
\[\{\gen{x_1},\gen{x_2,y},P_1,\gen{x_3,y}\}\]
induces a square, which is a contradiction. Therefore, $P_1\cong C_{p_1}\times C_{p_1}$ or $Q_8$.

Finally assume that $m=1$. Then $G\cong C_{p_1^kp_2}$ and $\F(G)$ is an star graph with $P_1$ at the center and $\{HP_2:H\leqslant\Phi(P_1)\}$ is the set of pendants.

(3) $n=1$. Then $G=P_1$. Let $\{x_1,\dots,x_m\}$ be a minimal generating set for $G$. Since $G=\gen{P'_1,x_1}\cdots\gen{P'_1,x_m}$, we get $m\leqslant3$ by Lemma \ref{mutuallypermutingsubgroups}. First suppose that $m=3$. We show that $\gen{x,y}$ is a maximal subgroup of $G$ for all $2$-element subsets $\{x,y\}$ of any minimal generating set of $G$. If not, let $\{x,y\}$ be a counterexample and $\{x,y,z\}$ be a generating set for $G$. Let $M$ be a maximal subgroup of $G$ containing $\gen{x,y}$ properly. Then 
\[\{\gen{x,y},\gen{G',z},M,\gen{G',xz}\}\]
induces a square, which is a contradiction. Now, assume that $m=2$. Clearly, a subgroup $H$ of $G$ is a vertex if and only if $\Phi(G)<H\Phi(G)<G$, that is, $H=\gen{K,h}$ for some subgroup $K$ of $\Phi(G)$ and element $h\in G\setminus\Phi(G)$. If $\Phi(G)=XY$ is a proper factorization such that $\gen{X,x}\cap\Phi(G)\subset\Phi(G)$ and $\gen{Y,y}\cap\Phi(G)\subset\Phi(G)$ for a generating set $\{x,y\}$ of $G$, then 
\[\{\gen{X,x},\gen{Y,y},\gen{\Phi(G),x},\gen{\Phi(G),y}\}\]
induces a square, which is impossible. Hence, for every proper factorization $\Phi(G)=XY$ of $\Phi(G)$ and every generating set $\{x,y\}$ of $G$ either $\gen{X,x}\supseteq\Phi(G)$ or $\gen{Y,y}\supseteq\Phi(G)$. We show that $\Phi(G)$ is cyclic. Suppose on the contrary that $m:=d(\Phi(G))>1$. Let $\max(\Phi(G))$ denote the set of all maximal subgroups of $\Phi(G)$. Then 
\[|\max(\Phi(G))|=\frac{p^m-1}{p-1}=p^{m-1}+\cdots+p+1.\]
Clearly, $\gen{g}$ acts on $\max(\Phi(G))$ and that $\Phi(G)$ has $g$-invariant maximal subgroups for all $g\in G$. Our assumption implies that $\Phi(G)$ has a maximal subgroup $M$ such that $M$ is the only $g$-invariant maximal subgroup of $\Phi(G)$ for all elements $g\in G\setminus\Phi(G)$. Thus, $\ol{G}$ acts fixed-point-freely on $\max(\Phi(G))\setminus\{M\}$ so that $|\max(\Phi(G))|-1$ is a multiple of $p^2$, a contradiction. Therefore, $\Phi(G)$ is cyclic.

One can easily see that the groups in (1)--(7) have square-free factorization graphs. The proof is complete.
\end{proof}

All non-abelian finite simple groups with proper factorizations, as well as all their factorizations into maximal subgroups, are classified by Liebeck, Praeger, and Saxl \cite{mwl-cep-js}. For instance, it is known that $PSL(2,p^n)$ has no proper factorizations whenever $p^n\equiv1\pmod 4$ and $p^n\notin\{5,9,29\}$ (see also \cite[Theorem 3.3]{fs-mfdg}). If $H$ is an arbitrary non-abelian finite simple group with no proper factorizations and $p$ is a prime, then a simple verification shows that $H\times\Z_p$ has only the proper factorizations $G=(H\times\{1\})K$, where $\pi_1(K)$ is a proper subgroup of $H$ and $\pi_2(K)=\Z_p$ with $\pi_i$ being the projection on $i$th component for $i=1,2$. Therefore, $\F^*(G)$ is an star graph, that is, $G$ satisfies the conditions in the first part of Theorem \ref{4-cycle}.
\begin{acknowledgment}
The authors are indebted to the referee whose suggestions and corrections improved the paper substantially.
\end{acknowledgment}

\end{document}